\documentclass{elsarticle}

\usepackage[T1]{fontenc}
\usepackage{amsthm,amsmath,amssymb}
\usepackage{graphicx}
\usepackage{mathabx}
\usepackage{tikz}
\usepackage{enumitem}
\usepackage{todonotes}
\usepackage{cleveref}
\usepackage{multirow}
\usepackage{stackengine}

\setlist[enumerate]{noitemsep, nolistsep}
\setlist[itemize]{noitemsep, nolistsep}

\newtheorem{theorem}{Theorem}
\newtheorem{lemma}[theorem]{Lemma}
\newtheorem{corollary}[theorem]{Corollary}

\theoremstyle{definition}
\newtheorem{definition}{Definition}

\begin{document}

\begin{frontmatter}
    \title{The Game Chromatic Number of Complete Multipartite Graphs}
    
    \author{Krzysztof Turowski\fnref{kt}}
    \address{Theoretical Computer Science Department, Jagiellonian University, Kraków, 30-348 Poland}
    \fntext[kt]{E-mail: krzysztof.szymon.turowski@gmail.com. Corresponding author}

    \author{Pawe\l~Obszarski\fnref{po}}
    \address{Department of Algorithms and System Modeling, Gdansk University of Technology, Gdańsk, 80-233 Poland}
    \fntext[po]{E-mail: obszar@eti.pg.gda.pl}
    
    \author{Hubert Zi\k{e}ba\fnref{hz}}
    \address{Theoretical Computer Science Department, Jagiellonian University, Kraków, 30-348 Poland}
    \fntext[hz]{E-mail: hub.zieba@student.uj.edu.pl}

    \begin{abstract}
        In this paper we investigate the game chromatic number for complete multipartite graphs, an important class of dense graphs.
        Our results consist of two findings:
        \begin{itemize}
          \item[(1)] building on Dunn's strategies for equal-sized parts, we determine the closed-form formulas for the game chromatic number for all complete multipartite graphs without singletons, with respective optimal strategies for Alice and Bob,
          \item[(2)] for an explicit list of part sizes, we give an exact algorithm that computes the game chromatic number in polynomial time and allows for reconstructing an optimal strategy for both players.
        \end{itemize}

        Additionally, we provide counterexamples to earlier, partial formulas for complete multipartite graphs with singletons and large equal-sized parts.
    \end{abstract}
    \begin{keyword}
        chromatic games \sep game chromatic number \sep complete multipartite graphs
    \end{keyword}
\end{frontmatter}

\section{Introduction}

The origins of the map-coloring game can be traced to a $1981$ article in Scientific American \cite{gardner}, but it has been analyzed
extensively only since it was reinvented by Bodlaender a decade later \cite{bodlaender} as a game played on graphs. Today,
there are many generalizations and variations of the graph coloring game, depending on what exactly is colored and what
the additional constraints on the graph structure, admissible coloring, etc. are. For more examples, see the survey \cite{map-coloring},
covering some variants, techniques, and results for these problems.

The standard version of the graph coloring game is played between Alice and Bob on a graph $G$ with a
set $C$ of $q$ colors, with $q$ fixed. We say that color $c\in C$ is \emph{legal} for a vertex $v\in V(G)$ if no neighbor
of $v$ is colored with $c$. A \emph{valid move} colors an uncolored vertex with a legal color. The game proceeds with Alice and Bob taking alternating turns, with Alice moving first, and making valid moves
until the entire graph is colored or some uncolored vertex has no legal color. Alice wins in
the former case and Bob in the latter.

The game chromatic number of a graph $G$, denoted by $\chi_g(G)$, is defined as the minimum $q$ such that there exists a
winning strategy for Alice, that is, it is certain that the entire graph will be colored regardless of Bob's strategy.
This parameter is well-defined because Alice always wins if $C$ contains at least as many colors as there are vertices of
$G$.

A strategy is \emph{well-defined} if it specifies a valid move at every turn of the player in a game that has not ended and in which that player has followed the strategy. We call a strategy \emph{optimal} for Alice on $G$ if it guarantees a win with $\chi_g(G)$ colors, and optimal for Bob if it guarantees a win with every smaller palette, regardless of the opponent's moves. A set of strategies for Alice is optimal on a class of graphs if, for each graph in the class, at least one of them is optimal. When a strategy permits several choices, its claimed guarantee applies to every permitted choice.

\subsection{Previous results}

The graph coloring game was studied by many authors.
In the case of forests $\mathcal{F}$ we know that $\max\{\chi_g(G)\colon G\in\mathcal{F}\}\leq 4$, and that this bound is tight \cite{kierstead-interval}. Polynomial algorithms are also known for deciding whether $\chi_g(F) = 2$ for a given forest $F$ \cite{toci2015game} and for finding the exact value of $\chi_g(G)$ for caterpillars \cite{furtado2019caterpillars}; however, the computational complexity of computing the value of $\chi_g(F)$ is still unknown.

For the class of planar graphs $\mathcal{P}$ it was proved
in \cite{kierstead-planar,zhu-planar} that $8\leq\max\{\chi_g(G)\colon G\in\mathcal{P}\}\leq 17$. For the class of outerplanar
graphs $\mathcal{OP}$, it was shown in \cite{zhu-outerplanar} that $6\leq\max\{\chi_g(G)\colon G\in\mathcal{OP}\}\leq 7$. For
$k$-trees $\mathcal{KT}$ it is the case that $\max\{\chi_g(G)\colon G\in\mathcal{OP}\}=3k+2$ for $k\geq 2$ (see
\cite{zhu-ktrees}). Similarly, it was shown in \cite{sidorowicz} that $\max\{\chi_g(G)\colon G\in\mathcal{C}\}=5$ for cacti
graphs $\mathcal{C}$. Finally, for interval graphs it was proved in \cite{kierstead-interval} that
$\chi_g(G)\leq 3\omega(G)-2$, where $\omega(G)$ is the clique number of $G$, and that there are examples of interval graphs with
$\chi_g(G)\geq 2\omega(G)-1$.

The bounds on the game chromatic number were also studied for various product graphs, most notably Cartesian product graphs \cite{peterin2007game,bartnicki2008game,zhu2008game,sia2009game}, direct product graphs \cite{alagammai2016game}, and most recently strong product graphs \cite{enomoto2023game}.

In another line of research, in \cite{dinski1999bound} the value of the game chromatic number for any graph can be bounded by a function of the acyclic chromatic number $\chi_a(G)$, i.e. the minimum number of colors such that there exists a coloring where each pair of colors induces an acyclic graph in $G$.
Moreover, for random graphs $G_{n,p}$ it was proved in \cite{bohman2008game,frieze2013game,keusch2014game} that with high probability the game chromatic number is within a multiplicative range of the chromatic number of the graph $\chi(G)$.

\subsection{Our results}

As is clear from the above survey, the graph coloring game was mostly studied for certain classes of graphs. In particular, the graphs were often very sparse and had a small value of the chromatic number.
In this paper we go in a different direction and focus on a dense class of graphs, but one that has a well-known value of the chromatic number, i.e.\ complete multipartite graphs.

Throughout this paper we denote by $K_{r_1, \ldots, r_k}$ a complete $k$-partite graph on $n = \sum_{i = 1}^k r_i$ vertices with a partition into disjoint independent sets $V_1$, $V_2$, \ldots, $V_k$ of cardinalities $r_1$, $r_2$, \ldots, $r_k$.
Without loss of generality, we assume that $r_1 \ge r_2 \ge \ldots \ge r_k$.
These independent parts are uniquely determined up to reordering. In particular, $k = 1$ gives an independent set with game chromatic number $1$, whereas $r_1 = \cdots = r_k = 1$ gives a complete graph with game chromatic number $k$.
Let us also denote by $l_j = |\{i \in \{1, \ldots, k\}\colon r_i = j\}|$ the count of sets $V_i$ of size exactly $j$.
Finally, we can assume that the colors appear in the game in a natural order starting with $1$.

The case of equal-sized parts was investigated by Dunn in \cite{dunn2012complete}. In the following statement we use his notation $K_{r*n}$ for a graph with $r$ parts, each of size $n$.
\begin{theorem}[\cite{dunn2012complete}, Theorem 1]
For positive integers $r$ and $n$, the graph $G = K_{r*n}$ satisfies
\begin{align*}
    \chi_g(G) =
    \begin{cases}
        r & \text{if $n = 1$,} \\
        2 r - 2 & \text{if $n = 3$ and $r \ge 3$,} \\
        2 r - 1 & \text{otherwise.}
    \end{cases}
\end{align*}
\end{theorem}
Thus Dunn's case $n = 1$ is a complete graph, while $r = 1$ gives an independent set. Outside this statement we retain $k$ for the number of parts, $r_i$ for their individual sizes, and $n$ for the total number of vertices, since the parts need not have equal sizes. Our fixed labels $V_i$, ordered by size, also differ from Dunn's labels $P_i$, which follow the order in which parts are first colored.
There was also a poster \cite{dunn-poster} giving formulas without proofs for complete $k$-partite graphs $K_{r_1,\ldots, r_k}$ with either all $r_i \in \{1, 2\}$ or all $r_i \in \{1, 3\}$, and reporting that the case with all $r_i \in \{1, 2, 3\}$ had been determined.
More recently, the authors of \cite{furtado2026coloring} studied singleton parts, including parts of sizes $1$, $2$, and $3$, as well as equal parts with larger size. They also provided a characterization of graphs satisfying $\chi_g(G) = \chi(G)$.

In this paper, we determine the game chromatic number for complete multipartite graphs.
First, in \Cref{sec:no_singletons} we present closed-form formulas for $\chi_g(G)$ for all complete multipartite graphs without singletons (see \Cref{tab:results} for a summary). Strategy $(A1)$ is Dunn's strategy from the proof of his Lemma~1, and $(A2)$ adapts his argument for parts of size $3$. Strategy $(B1)$ retains the response rule from his Lemma~2, with explicit choices among parts of unequal sizes. Our contribution is the analysis for unequal parts, including the odd-order strategy $(A3)$ and matching lower bounds that account for parts of size $2$ and $3$ and the parity of $n$. Together these give the exact formulas and an optimal choice among Alice's and Bob's strategies.
Next, in \Cref{sec:corrections} we show how singletons complicate the problem, leading to the failure of the above strategies. Additionally, we provide counterexamples to the partial formulas for the game chromatic number for complete multipartite graphs with singletons in Theorems~4.1(iv) and~4.2(ii) from \cite{furtado2026coloring}.
Finally, \Cref{sec:general} establishes the main result of the paper: although there is no known explicit formula in the general case, the value of $\chi_g(G)$ and optimal strategies for Alice and Bob may be computed in polynomial time. In particular, we show that the set of optimal moves for both Alice and Bob might be reduced to a constant number and, moreover, there exists a dynamic program on $O(k^5)$ states that reveals the equilibrium and best choices for both players.

\section{Case without singletons}\label{sec:no_singletons}

Let us first show a summary of closed-form solutions for the game chromatic number when a graph does not contain any singletons.
The upper bounds in \Cref{cor:a1,cor:a2,cor:a3} use $(A1)$, $(A2)$, and $(A3)$, respectively; all nontrivial lower bounds in the table follow from $(B1)$.

\begin{table}[htbp]
\centering\small
\caption{The summary of our results for the case without singletons}
\label{tab:results}
\begin{tabular}{|c|c|c|c|c|c|}
\hline
$k$ & $r_i$ & $n$ & $\chi_g(K_{r_1, \ldots, r_k})$ & Upper bound & Lower bound \\ \hline
$1$ & any & any & $1$ & \multicolumn{2}{c|}{obvious} \\ \hline
$\ge 2$ & $r_1 = 1$ & $k$ & $k$ & \multicolumn{2}{c|}{obvious} \\ \hline
$2$ & $r_2 = 1$ & any & $2$ & \Cref{cor:a1} & obvious \\ \hline
$2$ & $r_2 \ge 2$ & any & $3$ & \Cref{cor:a1} & \Cref{col:b1-main} \\ \hline
$\ge 3$ & \Centerstack{$\exists_jr_j=3$ $r_k=2$} & even & $2 k - 2$ & \Cref{cor:a2} & \Cref{col:b1-3} \\ \hline
$\ge 3$ & \Centerstack{$\forall_jr_j\neq3$ $r_k=2$} & even & $2 k - 1$ & \Cref{cor:a1} & \Cref{col:b1-2} \\ \hline
$\ge 3$ & \Centerstack{$\exists_jr_j=3$ $r_k=2$} & odd & $\min\left\{2 k - 2, \sum_j \left\lceil\frac{r_j}{2}\right\rceil\right\}$ & \Centerstack{\Cref{cor:a2} \Cref{cor:a3}} & \Cref{col:b1-3} \\ \hline
$\ge 3$ & \Centerstack{$\forall_jr_j\neq3$ $r_k=2$} & odd & $\min\left\{2 k - 1, \sum_j \left\lceil\frac{r_j}{2}\right\rceil\right\}$ & \Centerstack{\Cref{cor:a1} \Cref{cor:a3}} & \Cref{col:b1-2} \\ \hline
$\ge 3$ & $r_k = 3$ & any & $2 k - 2$ & \Cref{cor:a2} & \Cref{col:b1-3} \\ \hline
$\ge 3$ & $r_k \ge 4$ & any & $2 k - 1$ & \Cref{cor:a1} & \Cref{col:b1-main} \\ \hline
\end{tabular}
\end{table}

Finally, let us introduce several concepts that we will use throughout this paper. We will call a set with all vertices colored, no vertices colored, or some (but not all) vertices colored a \emph{fully colored}, \emph{uncolored}, or \emph{partially colored set}, respectively.
We will also say that a player \emph{started} coloring a set $V_i$ or, equivalently, that $V_i$ was started by that player if he or she colored the first vertex in this set.

Let us also denote the move after which every set $V_i$ has at least one vertex colored as the \emph{fixing move}. This concept is crucial in our investigations of the effectiveness of the strategies we proposed for Alice and Bob.

The importance of the above concepts for the chromatic game on complete multipartite graphs is directly supported by the following intuitive argument:
if Alice has a strategy such that she and Bob always use only at most $l$ colors up to a moment when all $V_i$ are partially or fully colored (i.e. until the fixing move), then a set of $l$ colors is sufficient to color the whole graph. Bob can always be forced to repeat the colors from this point on since by the structure of $K_{r_1, \ldots, r_k}$ if color $c$ is used in
any set $V_i$, then it remains feasible for all other vertices in this
set for the rest of the game -- and it is forbidden for all other vertices.
Therefore $\chi_g(K_{r_1, \ldots, r_k}) \le l$.

On the other hand, if Bob has a strategy such that the players always use $l$ colors while there still remains at least one uncolored set $V_i$ (i.e. before the fixing move), then $\chi_g(K_{r_1, \ldots, r_k}) > l$.

\subsection{Strategies for Alice}\label{sec:alice}

First, we will introduce two simple but powerful strategies for Alice:
\begin{definition}
  Let $(A1)$ be the following strategy for Alice: in any move, pick any vertex in any uncolored set $V_i$ and assign to it a new color. Otherwise, pick any vertex in any partially colored set $V_i$ and use a color that was already used for some vertex in $V_i$.
\end{definition}

In the examples below, $(A,i,c)$ or $(B,i,c)$ means that Alice or Bob colors an uncolored vertex of $V_i$ with color $c$.
For $(A1)$ on $K_{2,2,2}$, one possible play is
\[
  (A,1,1),\ (B,1,2),\ (A,2,3),\ (B,2,4),\ (A,3,5).
\]
The last move is the fixing move; Bob can then color the remaining vertex with color $5$.

\begin{definition}
  Let $(A2)$ be the following strategy for Alice for $K_{r_1, \ldots, r_k}$ with $k \ge 2$ and $r_j = 3$ for some particular (fixed) $j$:
  \begin{enumerate}
      \item in the first move pick an uncolored vertex from $V_j$ and assign to it a new color,
      \item otherwise, if Bob played in his last move at a vertex in $V_j$ and $V_j$ is still partially colored, then pick another uncolored vertex from $V_j$ and repeat the color just used by Bob,
      \item otherwise, if there is an uncolored set $V_i$, then pick any uncolored vertex from $V_i$ and assign a new color to it,
      \item otherwise, pick any uncolored vertex from a partially colored set $V_i$ and assign to it a color that was already used for some vertex in $V_i$.
  \end{enumerate}
\end{definition}
Note that the last two steps can be succinctly rephrased as ``otherwise, continue with $(A1)$''.

For $(A2)$ on $K_{4,3,2}$, fix $j = 2$. A possible play is
\[
  (A,2,1),\ (B,2,2),\ (A,2,2),\ (B,1,3),\ (A,3,4).
\]
Alice completes the triple on her second turn, forcing Bob to start another part. The fixing move uses only the fourth color. If Bob instead starts $V_1$ on his first turn, Alice can immediately start $V_3$ and make the fixing move.

Now we proceed to simple bounds on the number of colors necessary to finish the game on the whole graph:
\begin{theorem}
  \label{lem:a1}
  For any graph $K_{r_1, \ldots, r_k}$, Alice wins by using $(A1)$ with any set of at least $2 k - 1$ colors. If $k = 2$ and $r_2 = 1$, choosing the singleton first in $(A1)$ ensures a win with two colors.
\end{theorem}

\begin{proof}
  Let Alice use $(A1)$ with at least $2 k - 1$ colors. Clearly, after at most $2 k - 1$ moves in total ($k$ by Alice, $k - 1$ by Bob) every $V_i$ contains at least one colored vertex. Alice and Bob are then always able to finish coloring the rest of the graph, as they can always reuse colors that already appear in the same $V_i$.
  If $k = 2$ and $r_2 = 1$, Alice starts the singleton, so Bob must start the other part. Thus two colors suffice to reach the fixing move and finish the game.
\end{proof}

\begin{corollary}
  \label{cor:a1}
  For any graph $K_{r_1, \ldots, r_k}$ it holds that $\chi_g(K_{r_1, \ldots, r_k}) \le 2 k - 1$. If $k = 2$ and $r_2 = 1$, then $\chi_g(K_{r_1,1}) \le 2$.
\end{corollary}

It turns out that Alice can ensure saving one color if the graph contains at least one part with exactly three vertices:
\begin{theorem}
  \label{lem:a2}
  For any graph $K_{r_1, \ldots, r_k}$ with $k \ge 3$ and $r_j = 3$ for some $j \in \{1, 2, \ldots, k\}$, Alice wins by using $(A2)$ with $V_j$ as the fixed part and any set of at least $2 k - 2$ colors.
\end{theorem}

\begin{proof}
  Let Alice use $(A2)$ with at least $2 k - 2$ colors.
  The crucial observation is that before the fixing move, Bob had to start coloring at least one of the sets $V_i$:
  \begin{itemize}
      \item either Bob in his first move played also in $V_j$, so Alice colored the last vertex in $V_j$ in her second move -- and therefore Bob in his second move was forced to pick a vertex in an uncolored set $V_i$,
      \item or Bob in his first move played in some $V_i$ ($i \neq j$) -- but then again he was the first player to pick a vertex from some previously uncolored set.
  \end{itemize}
  Regardless of Bob's choice, it means that before the fixing move Alice started coloring at most $k - 2$ of the sets $V_i$ as one was started by Bob and another has to remain uncolored.
  Moreover, by the definition of the strategy, $V_j$ was the only set in which she played more than once before the fixing move -- so the total number of her moves before the fixing move is at most $k - 1$. Clearly, Bob too could not make more than $k - 1$ moves before the fixing move occurred.
  Therefore the number of colors used before the fixing move cannot exceed the sum of $k - 2$ (upper bound on the number of new colors introduced by Alice before the fixing move) and $k - 1$ (upper bound on the number of new colors introduced by Bob before the fixing move) -- that is, it is at most equal to $2 k - 3$.
  Therefore any set of at least $2 k - 2$ colors is sufficient to make the fixing move and to complete the whole coloring according to the rules, as the players can always reuse colors that already appear in the same $V_i$.
\end{proof}

\begin{corollary}
  \label{cor:a2}
  For any graph $K_{r_1, \ldots, r_k}$ with $k \ge 3$ and $r_j = 3$ for some $j \in \{1, 2, \ldots, k\}$ it holds that $\chi_g(K_{r_1, \ldots, r_k}) \le 2 k - 2$.
\end{corollary}

Now we introduce another strategy for complete multipartite graphs with an odd number of vertices. It is particularly well-suited for graphs with many $r_i = 2$:
\begin{definition}
  Let $(A3)$ be the following strategy for Alice for $K_{r_1, \ldots, r_k}$ with an odd number of vertices $n$:
  \begin{enumerate}
      \item in the first move pick an uncolored vertex from $V_i$ corresponding to the smallest odd $r_i$ and assign to it a new color,
      \item otherwise, if Bob played in his last move at a vertex in some partially colored set $V_j$, then pick another uncolored vertex from $V_j$ and repeat the color just used by Bob,
      \item otherwise, if there is a partially colored set $V_j$, then pick any uncolored vertex from $V_j$ and assign to it a color that was already used for some vertex in $V_j$,
      \item otherwise, pick any uncolored vertex from an uncolored set $V_j$ with the smallest odd $r_j$ and assign a new color to it.
  \end{enumerate}
\end{definition}

For $(A3)$ on $K_{3,2,2}$, a possible complete play is
\[
  (A,1,1),\ (B,1,2),\ (A,1,2),\ (B,2,3),\ (A,2,3),\ (B,3,4),\ (A,3,4).
\]
Alice starts the only odd part and then repeats Bob's colors. Bob makes the fixing move on his third turn, and four colors suffice.

Note that the strategy does not specify what to do if there are no partially colored sets and no odd uncolored sets. This is a deliberate decision on our part since we can prove that such a situation cannot occur:
\begin{lemma}
  \label{lem:a3-valid}
  Let $K_{r_1, \ldots, r_k}$ be a graph with an odd number of vertices $n$. Then $(A3)$ is a well-defined strategy for Alice. In particular, before each of her moves in a game that has not ended, there is at least one partially colored set or one odd uncolored set.
\end{lemma}

\begin{proof}
  Suppose that by using $(A3)$ Alice reached a state of the game when she faces only fully colored sets and even uncolored sets. This means that the total number of moves played is odd, as $n$ is odd. But if this is the case, then it has to be Bob's move, not Alice's -- a contradiction.
  Repeating a color in the same part is legal, and if an uncolored part remains in a game that has not ended, a new color is available. Thus every applicable rule specifies a valid move.
\end{proof}

Now we are ready to proceed with an assessment of the quality of this strategy:
\begin{theorem}
  \label{lem:a3}
  For any graph $K_{r_1, \ldots, r_k}$ with odd $n$, Alice wins by using $(A3)$ with any set of at least $\sum_i \left\lceil\frac{r_i}{2}\right\rceil$ colors.
\end{theorem}

\begin{proof}
  Let Alice use $(A3)$, which is well-defined by \Cref{lem:a3-valid}, with at least $\sum_i \left\lceil\frac{r_i}{2}\right\rceil$ colors.
  Observe that using $(A3)$ ensures that Alice never starts coloring a set with an even number of vertices. 
  Now, let $V_j$ be the set in which the fixing move is played.
  For any $i \neq j$ we know that at most $\left\lceil\frac{r_i}{2}\right\rceil$ colors are used:
  \begin{itemize}
      \item if $r_i$ is even, then Bob always plays the first move and there are at most $\frac{r_i}{2}$ moves by Bob in $V_i$ -- therefore, the total number of colors used in $V_i$ does not exceed $\frac{r_i}{2}$,
      \item if $r_i$ is odd and Bob colored the first vertex in $V_i$, then Bob played there at most $\frac{r_i + 1}{2}$ moves -- so he used at most $\frac{r_i + 1}{2}$ different colors (and Alice always repeated already used ones),
      \item if $r_i$ is odd and Alice colored the first vertex in $V_i$, then Bob played there at most $\frac{r_i - 1}{2}$ moves -- so they used at most $\frac{r_i + 1}{2}$ different colors.
  \end{itemize}
  In total, before the fixing move we need at most $\sum_{i \neq j} \left\lceil\frac{r_i}{2}\right\rceil$ colors for all vertices in all sets other than $V_j$.
  If at least $\sum_{i \neq j} \left\lceil\frac{r_i}{2}\right\rceil + 1$ colors are available, then Alice or Bob can always play the fixing move and finish the whole coloring. Since $\left\lceil\frac{r_j}{2}\right\rceil \ge 1$, the stated palette is sufficient.
\end{proof}

\begin{corollary}
  \label{cor:a3}
  For any graph $K_{r_1, \ldots, r_k}$ with odd $n$ it holds that $\chi_g(K_{r_1, \ldots, r_k}) \le \sum_i \left\lceil\frac{r_i}{2}\right\rceil$.
\end{corollary}

\subsection{Strategy for Bob}\label{sec:bob}

The following refinement of Dunn's response strategy suffices for Bob in all cases without singletons:
\begin{definition}
  Let $(B1)$ be the following strategy for Bob for $K_{r_1, \ldots, r_k}$:
  \begin{enumerate}
      \item if Alice picked a vertex in $V_i$ and $V_i$ is partially colored, pick any uncolored vertex from $V_i$,
      \item otherwise, if Alice picked a vertex in $V_i$ that is now fully colored, and there is any partially colored set, then choose any partially colored $V_j$ with the smallest number of uncolored vertices and pick any uncolored vertex from $V_j$,
      \item otherwise, if Alice picked a vertex in $V_i$ that is now fully colored, and there is no partially colored set, then choose any vertex from the largest uncolored $V_j$.
  \end{enumerate}
  Assign to the chosen vertex a new color if you can. Otherwise, reuse a color that was already used for some other vertex in the same set as the chosen vertex.
\end{definition}

For $(B1)$ on $K_{4,3,2}$, consider the moves
\[
  (A,2,1),\ (B,2,2),\ (A,2,1),\ (B,1,3).
\]
Bob first responds in Alice's part. Once Alice completes it, he starts the largest uncolored part, $V_1$. To illustrate the second rule, on $K_{4,3,3,2}$ the moves
\[
  (A,1,1),\ (B,1,2),\ (A,2,3),\ (B,2,4),\ (A,3,5),\ (B,3,6),\ (A,3,5)
\]
leave $V_1$ and $V_2$ partially colored, with two and one uncolored vertices, respectively. Bob chooses $V_2$ and uses a new color if one is available.

Now we proceed with the counterparts of \Cref{lem:a1,lem:a2,lem:a3}, establishing the optimality of the respective strategies for Alice and Bob. In the lower-bound arguments, colors are counted up to the fixing move assuming a new color remains available whenever required; if the palette is exhausted earlier while a part is uncolored, Bob has already won.
\begin{theorem}
  \label{lem:b1-main}
  Suppose that $K_{r_1, \ldots, r_k}$ has either $r_k \ge 4$, or $k = 2$ and $r_2 \ge 2$. If Bob uses $(B1)$, then he and Alice cannot color the whole graph $K_{r_1, \ldots, r_k}$ using any set with fewer than $2 k - 1$ colors.
\end{theorem}

\begin{proof}
  If $k = 2$ and $r_2 \ge 2$, the first rule of $(B1)$ makes Bob respond to Alice's first move in the same part with a second color. With only two colors, the other part has no legal color, proving the required lower bound $3 = 2 k - 1$. We may therefore assume $r_k \ge 4$ from now on in the proof.

  Note that using this strategy we can ensure that at any time there can exist at most one set $V_i$ that has two properties: $(a)$ it has exactly one vertex colored, and $(b)$ it was colored only by Bob. And if such a $V_i$ does exist, then by the definition of the strategy there also has to be somewhere in the graph a fully colored set $V_j$ with its first vertex colored by Alice, because otherwise Bob would never deliberately start coloring any set.
  Note that such a $V_j$ would have been started by Alice. Moreover, the fact that $r_j \ge r_k \ge 4$ together with Bob's strategy implies that Bob played at least $2$ moves there -- so in total in $V_j$ the players would use at least $3$ colors.
  
  Moreover, after every move by Bob until the fixing move all other partially or fully colored sets $V_l$, $l \notin \{i, j\}$, are:
  \begin{itemize}
      \item either started by Alice with an immediate response by Bob,
      \item or started by Bob with later moves by both Alice and Bob,
      \item or started by Bob with a second move also made by Bob.
  \end{itemize}
  In all cases, it is obvious that in these sets the players used at least $2$ distinct colors.
  
  Let us now think about a situation just before the fixing move.
  If Alice makes the fixing move, then there are two possible situations:
  \begin{itemize}
      \item either there exists some $V_i$ such that it has only one vertex colored and it was colored only by Bob and therefore the players used one color for this set, but also at least $3$ colors for the respective $V_j$ by the argument in the first paragraph, and at least $2$ colors for all sets $V_l$ ($l \notin \{i, j\}$), by the argument above.
      \item or there is no $V_i$ with this property -- and therefore every set is either started by Alice (with an immediate response by Bob), or started by Bob (with at least one more move by Bob since $r_j \ge 4$); in each case, at least $2$ colors are used in each set.
  \end{itemize}
  Either way, before the fixing move there were at least $2 k - 2$ colors used in total.
  
  If this is Bob's move, then all sets but the last one were already fully colored before the fixing move. But since $r_i \ge 4$, by his strategy Bob played in all sets but one at least twice, so he himself had to use at least $2 k - 2$ colors.

  Overall, Bob may ensure by using $(B1)$ that at least $2 k - 2$ colors were used before the fixing move. Therefore if the color set contains fewer than $2 k - 1$ colors, then there always exists a way to enforce a situation in which we cannot play the fixing move, i.e. legally extend the coloring to the last set, and therefore to the whole graph.
\end{proof}

\begin{corollary}
  \label{col:b1-main}
  For $K_{r_1, \ldots, r_k}$ with either $r_k \ge 4$, or $k = 2$ and $r_2 \ge 2$, it holds that $\chi_g(K_{r_1, \ldots, r_k}) \ge 2 k - 1$.
\end{corollary}

\begin{theorem}
  \label{lem:b1-2}
  Let $K_{r_1, \ldots, r_k}$ be such that $k \ge 3$ and $r_j \neq 3$, $r_j \ge 2$ for all $j \in \{1, 2, \ldots, k\}$. If Bob uses the strategy $(B1)$, then the players cannot color the whole graph $K_{r_1, \ldots, r_k}$ using any set with at most $\min\{2 k - 2, \sum_i \left\lceil\frac{r_i}{2}\right\rceil - 1\}$ colors.
  
  Moreover, if $n$ is even, Bob can ensure that he and Alice cannot color the whole graph $K_{r_1, \ldots, r_k}$ using any set with at most $2 k - 2$ colors.
\end{theorem}

\begin{proof}
  We split the proof into three subcases $(1)$--$(3)$, depending on the properties of the last set that was first colored by Bob before the fixing move occurred.
  
  \textbf{Case (1):} There is no such set, i.e. Alice started coloring all sets until the fixing move. In this case, using $(B1)$ guarantees that before the fixing move Alice used $k - 1$ different colors when she colored the first vertices of the first $k - 1$ sets $V_j$ with $r_j \ge 2$ and Bob used $k - 1$ different colors in his subsequent moves.
  Therefore, both players used at least $2 k - 2$ colors before reaching the fixing move, so it clearly holds that this number of colors is not sufficient to make a legal fixing move.
  
  \textbf{Case (2):} Assume that there is such a set $V_i$ with $r_i > 2$. Note that the strategy implies two facts: that Bob never started coloring any set $V_j$ with $r_j = 2$, and that, since $r_i > 2$, when Bob started coloring $V_i$ there could not be any partially colored sets.
  Moreover:
  \begin{itemize}
      \item when Bob started coloring $V_i$, the sum of the cardinalities of all fully colored sets had to be odd, so at least one fully colored $V_j$ had an odd number of vertices -- but then $r_j \ge 5$, so there are at least $\lceil\frac{r_j}{2}\rceil \ge 3$ distinct colors used there (one by the player starting it and at least $\lfloor\frac{r_j}{2}\rfloor$ by subsequent moves by Bob due to his strategy),
      \item if Alice started coloring any $V_j$ with $r_j \ge 4$ ($j \neq i$), then Bob responded in the same set in his next move, so they used at least $2$ different colors there,
      \item if Bob started coloring any $V_j$ with $r_j \ge 4$ ($j \neq i$), then it had to be fully colored by the time he started coloring $V_i$ -- in particular, this implies that he played at least $\lfloor\frac{r_j}{2}\rfloor \ge 2$ moves there, all using different colors,
      \item since Alice started coloring all $V_j$ with $r_j = 2$, by his strategy Bob immediately responded by coloring the other vertex from the same set, so they used $2$ distinct colors there.
  \end{itemize}
  In total, just before the fixing move the players used at least one color for one set (i.e. $V_i$ itself), at least $3$ colors for another set (i.e. $V_j$ from the first case above), and at least $2$ colors for all other sets (excluding the set with the fixing move). By this, we arrive at the result that they used at least $2 k - 2$ colors while there still remained one uncolored set -- which contradicts the possibility of making a legal fixing move.

  \textbf{Case (3):} Suppose we have an appropriate set $V_i$ with $r_i = 2$.
  This means that all sets with $r_j \neq 2$ had to be fully colored before Bob started coloring $V_i$.
  But then we know that by Bob's strategy each set $V_j$ with $r_j > 2$ was:
  \begin{itemize}
      \item either started by Alice, with at least $\lfloor\frac{r_j}{2}\rfloor$ moves by Bob,
      \item or started by Bob, with at least $\lceil\frac{r_j}{2}\rceil$ moves by Bob.
  \end{itemize}
  In both cases the players use at least $\lceil\frac{r_j}{2}\rceil$ distinct colors for the vertices in $V_j$.
  In total, all sets with $r_j > 2$ require at least $\sum_{j: r_j > 2} \lceil\frac{r_j}{2}\rceil$ distinct colors.
  Clearly, each set with $r_j = 2$ requires at least one color -- so to color all but the last one (in which some player makes the fixing move) we need $\sum_{j: r_j = 2} \lceil\frac{r_j}{2}\rceil - 1$ colors.
  Therefore, Bob can ensure that at least
  \begin{align*}
      \sum_{j: r_j > 2} \left\lceil\frac{r_j}{2}\right\rceil + \sum_{j: r_j = 2} \left\lceil\frac{r_j}{2}\right\rceil - 1 = \sum_{j} \left\lceil\frac{r_j}{2}\right\rceil - 1
  \end{align*}
  colors are needed before the fixing move. And if the size of the set of colors does not exceed that number, no player can make a legal fixing move.
  
  Finally, note that the case $(3)$ occurs only for odd $n$, as it implies that all $V_j$ with $r_j > 2$ were fully colored before Bob played his first move in $V_i$. However, if $n$ is even, the situation when only sets $V_j$ with $r_j = 2$ are uncolored and there are no partially colored sets can only occur before Alice's move.

  To conclude the proof for odd $n$ it is sufficient to observe that if in some cases $c_1$ colors are insufficient and in all other cases $c_2$ colors are insufficient to make the legal fixing move, then $\min\{c_1, c_2\}$ colors are always insufficient.
\end{proof}

\begin{corollary}
  \label{col:b1-2}
  For $K_{r_1, \ldots, r_k}$ with $k \ge 3$ and $r_j \neq 3$, $r_j \ge 2$ for all $j \in \{1, 2, \ldots, k\}$ it holds that $\chi_g(K_{r_1, \ldots, r_k}) \ge \min\{2 k - 1, \sum_i \left\lceil\frac{r_i}{2}\right\rceil\}$.
  Additionally, if $n$ is even, then $\chi_g(K_{r_1, \ldots, r_k}) \ge 2 k - 1$.
\end{corollary}

As we proved earlier, the existence of a part of size $3$ in the graph helps Alice save one color, provided she uses a special strategy $(A2)$ instead of $(A1)$.
However, it turns out that the optimal strategy for Bob remains the same in this case: 
\begin{theorem}
  \label{lem:b1-3}
  Let $K_{r_1, \ldots, r_k}$ have $k \ge 3$, $r_k \ge 2$, and $r_j = 3$ for some $j \in \{1, 2, \ldots, k\}$. If Bob uses the strategy $(B1)$, then he wins on the graph $K_{r_1, \ldots, r_k}$ for any set with at most $\min\{2 k - 3, \sum_i \left\lceil\frac{r_i}{2}\right\rceil - 1\}$ colors.
  
  Moreover, if $n$ is even, Bob can ensure that he and Alice cannot color the whole graph $K_{r_1, \ldots, r_k}$ using any set with at most $2 k - 3$ colors.
\end{theorem}

\begin{proof}
  As in the previous proof, we distinguish the cases $(1)$--$(2)$ depending on the size of the last set started by Bob before the fixing move, denoted by $V_i$.

  \textbf{Case (1):} If either $V_i$ does not exist or it exists and its respective $r_i > 2$, then we know that before the fixing move Bob started coloring only sets with size greater than $2$.
  Let us call a set a \emph{B-singleton} if it has three properties: it is partially colored, it has exactly one vertex colored, and it was colored only by Bob.
  First, we note that after any move by Bob, there always exists at most one B-singleton.
  Clearly, after his first move this condition is met. Moreover, if it is true after his $l$-th move, then:
  \begin{itemize}
      \item either Alice starts coloring a new set $V_i$, in which case Bob responds in the same set, so overall no new B-singleton appears,
      \item or Alice plays in some set that was already partially colored, but then Bob starts coloring a new set $V_i$ only in the case that there are no partially colored sets -- so only if before his move there are no B-singletons.
  \end{itemize}
  In both cases, the invariant is also true after Bob's $(l + 1)$-th move.
  
  For any set $V_j$ other than the B-singleton and the set in which the fixing move is made, one of the following conditions holds:
  \begin{itemize}
      \item the first vertex was colored by Alice and the second one was colored right after that by Bob,
      \item the first vertex was colored by Bob; then at some point the second vertex was colored by Alice and the third one was colored right after that by Bob,
      \item the first vertex was colored by Bob and then, since the set is not a B-singleton, at some point the second vertex was also colored by Bob.
  \end{itemize}
  Thus, Bob's strategy guarantees that for every such $V_j$ there are at least $2$ distinct colors used. Of course, in a B-singleton exactly one color is used.

  Therefore, the total number of colors in use before the fixing move is at least equal to $1 + 2 (k - 2)$. Therefore in this case we cannot play the fixing move and complete the coloring if the set of colors contains strictly fewer than $2 k - 2$ colors.
  
  \textbf{Case (2):} If such a $V_i$ does exist and its respective $r_i = 2$, then we know that all $V_j$ with odd $r_j \neq 2$ had to be fully colored to force Bob to play the first move in $V_i$, so it implies that $n$ has to be odd.
  Here we basically repeat the arguments from the previous proof: Bob's strategy ensures that in each set $V_j$ with $r_j > 2$ the players use at least $\lceil\frac{r_j}{2}\rceil$ colors and in each set $V_j$ with $r_j = 2$ the players use at least one color, i.e. also at least $\lceil\frac{r_j}{2}\rceil$ colors.
  This, combined with the fact that the fixing move in this case is always played in a set $V_j$ with $r_j = 2$, guarantees that at least $\sum_{j} \lceil\frac{r_j}{2}\rceil - 1$ colors are needed before the fixing move.
  
  Thus, if $n$ is odd, then more than $\min\{2 k - 3, \sum_{j} \lceil\frac{r_j}{2}\rceil - 1\}$ colors are needed to complete the coloring. On the other hand, if $n$ is even, then only the first case may occur, so more than $2 k - 3$ colors are needed.
\end{proof}

\begin{corollary}
  \label{col:b1-3}
  For $K_{r_1, \ldots, r_k}$ with $k \ge 3$, $r_k \ge 2$, and $r_j = 3$ for some $j \in \{1, 2, \ldots, k\}$ it holds that $\chi_g(K_{r_1, \ldots, r_k}) \ge \min\{2 k - 2, \sum_i \left\lceil\frac{r_i}{2}\right\rceil\}$.
  
  Additionally, if $n$ is even, then $\chi_g(K_{r_1, \ldots, r_k}) \ge 2 k - 2$.
\end{corollary}

\begin{proof}
  By \Cref{lem:b1-3}, Bob wins with at most $\min\{2 k - 3, \sum_i \lceil r_i/2\rceil - 1\}$ colors, so Alice needs at least one more color. In particular, for even $n$, the insufficiency of $2 k - 3$ colors gives the lower bound $2 k - 2$.
\end{proof}

\section{Problems with closed-form formulas for the general case}
\label{sec:corrections}

When we allow for the appearance of singletons in a complete multipartite graph, the problem becomes more complicated.
The strategies presented above can still be optimal for some graphs; nevertheless, this claim no longer holds in general.

For example, let us consider a graph $K_{r, r, 1}$ (i.e. with $k = 3$) for $r \ge 5$.
Strategy $(A1)$ permits Alice to start both large parts before the singleton; against $(B1)$, this uses five colors up to the fixing move. Strategy $(A3)$ starts the singleton, but then keeps responding in the large part started by Bob. Against $(B1)$, this requires $\lceil r/2\rceil + 2 \ge 5$ colors up to the fixing move.
Moreover, the strategy $(A2)$ is inapplicable in this case, since $K_{r, r, 1}$ does not contain any set of size exactly $3$.
On the other hand, it is clear that the optimal strategy for Alice is to color the singleton first, let Bob begin coloring $V_1$ or $V_2$, and then fix the total number of colors at $3$ by playing the fixing move at any vertex from the remaining uncolored set and coloring it using a new color.
Thus, since it is straightforward that Alice cannot win with just two colors, it follows that $\chi_g(K_{r, r, 1}) = 3$.

Similarly, the strategy $(B1)$ fails to achieve the optimal coloring for $K_{2, 2, 1, \ldots, 1}$ with an even $k \ge 4$, since it requires using at most $k$ colors: when Alice starts by coloring a singleton, Bob responds by coloring a vertex in $V_1$ (or $V_2$, without loss of generality). Then Alice uses the same color on the other vertex in $V_1$, Bob picks a vertex from $V_2$ and colors it using a new color, Alice responds again by coloring the other vertex in $V_2$ with the same color as Bob, and they just have to keep coloring the singletons with new colors.
It can be easily verified that $\chi_g(K_{2, 2, 1, \ldots, 1}) = k + 1$.

A further difficulty in extending \Cref{tab:results} is that increasing a part size can decrease the game chromatic number. For example,
\[
  \chi_g(K_{3,2,2,2,1})=7,
  \qquad
  \chi_g(K_{4,2,2,2,1})=6.
\]
The first value follows from \cite[Theorem~2.4(i)]{furtado2026coloring}. For the second, the odd order permits Alice to use $(A3)$, and \Cref{cor:a3} gives the upper bound $2+1+1+1+1=6$. Bob forces at least six colors: if Alice starts a non-singleton part, he immediately adds a second color there; if she starts the singleton, he starts the part of size $4$ and adds a second color there on his next turn. Alice cannot complete that part or make the fixing move before his second turn. Thus one of the five parts must use two colors.

Unfortunately, it is even hard to get the proper formulas for partial cases. In particular, we would like to point out counterexamples to some results from \cite{furtado2026coloring}.
For brevity, let us introduce the notation $K_{1^s,r^t}$ for the complete multipartite graph with $s$ singleton parts and $t$ parts of size $r$.
We noticed that Theorem~4.1(iv) of \cite{furtado2026coloring} assigns $11$ to $K_{1^4,5^3}$, while Theorem~4.2(ii) and its worked example on page~6 of the paper assign $13$ to $K_{1^5,6^3}$. Both values are too large, as the following theorem shows, establishing $\chi_g(K_{1^4,5^3}) = 10$ and $\chi_g(K_{1^5,6^3}) = 12$.

\begin{theorem}
  \label{prop:counterexamples}
  For every integer $s \ge 3$ it holds that $\chi_g(K_{1^s,(s+1)^3}) = 2s + 2$.
\end{theorem}

\begin{proof}
  First, let us assume that in the first phase of her strategy Alice colors a singleton on every turn while there remains at least one uncolored.
  If the fixing move has already occurred in this phase, at most $2s$ colors have been used up to that move.

  Otherwise, let $h$ be the number of singletons colored by Alice when she first encounters no uncolored singleton remaining. By this moment, both players have made $h$ moves, and at most $2h$ colors have been used.
  Now we distinguish three cases:
  \begin{enumerate}
    \item[(1)] \textbf{$h < s$ and Bob has colored a non-singleton vertex:} at most two large parts (with size $s + 1$) remain uncolored. Alice colors a vertex in each of them on successive turns, so, including Bob's move in between, at most three additional colors are used until the fixing move, for a total at most $2 h + 3 \le 2 s + 1$.
    \item[(2)] \textbf{$h < s$ and Bob has colored only singleton vertices:} obviously $2 h = s$. Alice colors a vertex in a new large part on each turn, using, together with Bob, at most $5$ additional colors until the fixing move, for a total at most $2 h + 5 = s + 5 \le 2 s + 2$.
    \item[(3)] \textbf{$h = s$:} Bob has colored $s$ vertices in the large parts. If he has started at least two parts, Alice makes the fixing move with at most one additional color. Otherwise, all $s$ vertices lie in one part of size $s + 1$. Alice colors its last vertex with an existing color. Bob must then start a second larger part, and Alice immediately starts the third. These three moves introduce only two new colors, giving at most $2 s + 2$ in total.
  \end{enumerate}
  Thus, there exists a strategy for Alice ensuring the use of at most $2 s + 2$ colors.

  For the lower bound, let Bob use $(B1)$. We count colors through the fixing move, assuming that the palette has not already been exhausted while a part remains uncolored.

  Assume that Bob starts a singleton at some point and consider the first such move. By the definition of $(B1)$, all three larger parts are already fully colored. Before that move Bob has made at least $\lfloor\frac{3(s + 1)}{2}\rfloor$ moves, at most three of which were first moves in their respective parts. Hence at least $\lfloor\frac{3(s + 1)}{2}\rfloor - 3 \ge s - 1$ of his moves introduced additional colors in already started parts. Together with the $s + 3$ colors necessary to start all parts, this ensures that at least $2 s + 2$ colors have to be used.

  Otherwise, let $x \le 3$ be the number of parts started by Bob, let $y$ be his number of moves in already started parts, and let $z$ be Alice's number of such moves, all counted through the fixing move. Additionally, put $\delta=1$ if Alice makes the fixing move and $\delta=0$ if Bob does. Since Alice played first moves in exactly $s + 3 - x$ parts, she made a total of $s + 3 - x + z$ moves until the fixing move. Bob played $x + y$ moves, and from the fact that the players alternate and their numbers of moves differ by $\delta$ we get $y = s + 3 - 2 x + z - \delta$.

  If $x \le 1$, this immediately gives $y \ge s$. If $x \ge 2$, before Bob starts his $x$-th larger part, the preceding $x-1$ larger parts that he started must be fully colored. Each required $s$ moves after its first vertex was colored, so $y+z \ge s(x-1)$. Consequently,
  \begin{align*}
    2y = s+3-2x-\delta+(y+z)
       \ge x(s-2)+3-\delta \ge 2s-2.
  \end{align*}
  Thus $y \ge s-1$ in every case, and at least $(s + 3) + (s - 1) = 2 s + 2$ colors are required, establishing a lower bound and completing the proof.
\end{proof}

The crucial difference from the strategies used in \cite{furtado2026coloring} is that the upper-bound strategy completes a nearly colored part using an existing color. This changes whose turn it is when the next part must be started, a possibility omitted from the counting in Theorem~4.1. Also, let us note that for $K_{1^5,6^3}$ Theorem~4.2(ii) provides an erroneous value; however, this graph also satisfies the hypotheses of Theorem~4.1(iv), whose expression $A+2D+1$ gives the proper solution $12$ for $A=5$ singletons and $D=3$ larger parts.

In the proof of Theorem~4.2, $A$ Bob moves need not occupy exactly $\lceil A/r\rceil$ larger parts: he may distribute them among more parts. The two fixed policies considered there also omit Alice's adaptive response above. Turn counts require care: immediately after Alice's $A$-th move, Bob has made $A-1$ moves; counting $A$ includes his subsequent response. Our proof explicitly uses the latter boundary.

Unfortunately, we were not able to find closed-form formulas that would fix the problems presented above and allow us to provide refined versions of Theorems~4.1 and 4.2 from \cite{furtado2026coloring}.
Instead, in the next section we provide a polynomial-time algorithm that computes $\chi_g(K_{r_1, \ldots, r_k})$ and allows both Alice and Bob to find their optimal moves.

\section{Polynomial-time algorithm for the general case}
\label{sec:general}

We now consider arbitrary complete multipartite graphs, allowing any number of singletons. Our aim is to compute both the game chromatic number and optimal moves for Alice and Bob. First, we provide a compact representation of a position in a game as the sizes of its uncolored parts and the total number of uncolored vertices in parts already started. Then we show that only a constant number of moves need to be considered at each position. Finally, we use these observations to construct a dynamic program with $O(k^5)$ states.

\subsection{A compact representation}

Recall that after the fixing move both players can always finish coloring the graph using colors already in use. Thus, let us instead consider first an auxiliary game with an unlimited supply of colors, stopped immediately after the fixing move. Its value is the total number of distinct colors used, which Alice minimizes and Bob maximizes.

Let $U$ be the multiset of sizes of uncolored parts, and let $p$ be the total number of uncolored vertices in parts that contain at least one colored vertex. We call these $p$ vertices the \emph{pool}. A state of the game may be represented succinctly by a pair $(U, p)$. A move either starts a part of size $r \in U$, moving to the state $(U \setminus \{r\}, p + r - 1)$, or colors a vertex from the pool, moving to $(U, p - 1)$. Here and below, we assume that multiset subtraction removes one occurrence, and multiset union adds occurrences of the same values.

\begin{lemma}
  \label{lem:numerical-state}
  In the auxiliary game, Alice may always reuse a color when playing in the pool, and Bob may always use a new color. The value of the remaining play depends only on $U$, $p$, and the player to move.
\end{lemma}

\begin{proof}
  Obviously, a color used in one part cannot occur in another part. Every vertex in the pool can reuse a color already present in its part, or receive a new color. Both choices give the same successor state representation, only with the latter introducing an additional color. Since the supply of colors is unlimited, this choice imposes no restriction on subsequent moves. Thus reusing a color is optimal for Alice, and using a new color is optimal for Bob.

  The choice of a vertex within the pool changes neither the available types of subsequent moves nor their contributions to the number of colors. Starting a part of size $r$ always introduces one color and adds $r-1$ vertices to the pool. By induction on the number of uncolored vertices, the remaining value therefore depends only on the stated data.
\end{proof}

Each part in $U$ requires one color when its first vertex obtains a color. Subtract these $|U|$ necessary colors from the value of the remaining play, and denote the resulting values by $\mathcal{A}(U,p)$ and $\mathcal{B}(U,p)$ when Alice and Bob move next, respectively. In this notation, starting a part (i.e. moving from $U$ to some $U'$) and Alice's moves in the pool contribute zero, whereas each of Bob's moves in the pool contributes one. These definitions also apply to arbitrary numerical states with $p \ge 0$.

The fixing move ends the auxiliary game, so
\begin{equation}
  \label{eq:pool-terminal}
  \mathcal{A}(\emptyset,p) = \mathcal{B}(\emptyset,p) = 0.
\end{equation}
For $U \neq \emptyset$, \Cref{lem:numerical-state} gives
\begin{align}
  \mathcal{A}(U,p) &= \min\left(
    \{\mathcal{B}(U \setminus \{r\},p+r-1) \colon r \in U\}
    \cup \{\mathcal{B}(U,p-1) \colon p>0\}\right), \label{eq:pool-alice}\\
  \mathcal{B}(U,p) &= \max\left(
    \{\mathcal{A}(U \setminus \{r\},p+r-1) \colon r \in U\}
    \cup \{1+\mathcal{A}(U,p-1) \colon p>0\}\right). \label{eq:pool-bob}
\end{align}
Note that each move decreases the potential $p + \sum_{r \in U} r$ by one, so these recurrences are well-defined.

Now let us proceed with a relation between the coloring game and our auxiliary game on complete multipartite graphs:
\begin{theorem}
  \label{thm:numerical-game}
  For every complete multipartite graph,
  \[
    \chi_g(K_{r_1,\ldots,r_k}) = k + \mathcal{A}(\{r_1,\ldots,r_k\},0).
  \]
\end{theorem}

\begin{proof}
  By \Cref{lem:numerical-state}, the right-hand side is the value of the auxiliary game with Alice moving first. We show that its value is also the minimum size of a palette with which Alice wins the original game.

  Suppose Alice can ensure that at most $q$ colors are used through the fixing move in the auxiliary game. She follows the same strategy with $q$ available colors. This cannot exhaust the palette while some part remains uncolored: any completion to the fixing move would then require at least one more color, contradicting her guarantee for the auxiliary game. Therefore, every move prescribed by the strategy is legal until the fixing move, after which all vertices can be colored using colors already present in their parts.

  Conversely, if Bob can force more than $q$ colors until the fixing move in the auxiliary game, he just needs to follow the same strategy for a game with a palette of $q$ colors as long as a new color is available. Before a $(q+1)$-st color could be introduced, all $q$ colors must have been used and some part still remains uncolored. No vertex in that part has a legal color, so Bob has already won. This completes the proof.
\end{proof}

\subsection{Reducing the available moves}

Now we proceed with a series of lemmas establishing bounds between solutions for similar games.

\begin{lemma}
  \label{lem:pool-marked}
  For $\mathcal{T} \in \{\mathcal{A},\mathcal{B}\}$, every multiset $U$, and integers $r \ge 1$, $p \ge 0$, it holds that $\left|\mathcal{T}(U \cup \{r\}, p)-\mathcal{T}(U, p + r)\right| \le 1$.
\end{lemma}

\begin{proof}
  Let us distinguish a part of size $r$ in the first game, and distinguish a group of $r$ pool vertices in the second. Until a player first chooses a distinguished vertex, the two games have the same moves and contributions to the excess number of colors.

  If such a move occurs, the successor states are identical, since a move that removes the distinguished part from $U$ also adds $r - 1$ vertices to the pool. For Alice this move contributes zero in either game. For Bob it contributes zero when starting the distinguished part, and one when choosing a distinguished pool vertex. The values of these continuations therefore differ by at most one.

  There is also a possibility that all non-distinguished parts in $U$ are started before any distinguished vertex is chosen. The second game then ends. In the first game only the distinguished part remains uncolored. Alice can start it immediately, so its remaining value is zero on her turn. On Bob's turn the remaining value is at most one, since after one pool move Alice can start the last part.

  Note that the values assigned to corresponding leaves of both game trees differ by at most one, and taking minima or maxima preserves this bound.
\end{proof}

\begin{corollary}
  \label{cor:bob-pool}
  If $U \neq \emptyset$ and $p>0$, then $\mathcal{B}(U,p) = 1+\mathcal{A}(U,p-1)$.
  In particular, Bob has an optimal move in the pool whenever it is nonempty.
\end{corollary}

\begin{proof}
  For every $r \in U$, \Cref{lem:pool-marked} applied to $U \setminus \{r\}$ and $p-1$ gives $\mathcal{A}(U \setminus \{r\},p+r-1) \le 1+\mathcal{A}(U,p-1)$.
  However, the sides of this inequality are exactly the values from \Cref{eq:pool-bob} -- starting some uncolored part and making a pool move, respectively.
  Obviously, Bob prefers picking the latter, as it always leads to the maximized $\mathcal{B}(U,p)$.
\end{proof}

Next, we investigate the effect of adding two vertices.
\begin{lemma}
  \label{lem:pool-two}
  For every multiset $U$, $p \ge 0$, and $\mathcal{T} \in \{\mathcal{A},\mathcal{B}\}$, it holds that $0 \le \mathcal{T}(U,p+2)-\mathcal{T}(U,p) \le 1$.
\end{lemma}

\begin{proof}
  First, consider Alice's turn. If $|U| \le 1$, then $\mathcal{A}(U,p)=0$, since she can make the fixing move immediately whenever a part remains uncolored. If $|U| \ge 2$, \Cref{eq:pool-alice,cor:bob-pool} imply
  \[
    \mathcal{A}(U,p+2) = 1+\min\left\{
      \mathcal{A}(U,p),\ \min_{r \in U} \mathcal{A}(U \setminus \{r\},p+r)
    \right\}.
  \]
  By \Cref{lem:pool-marked}, every term in the inner minimum is at least $\mathcal{A}(U,p)-1$. Thus, the term $\mathcal{A}(U,p)$ is the one matching the upper bound, proving both inequalities for Alice.

  For Bob, the case $U=\emptyset$ is immediate. If $p>0$, \Cref{cor:bob-pool} reduces the claim to the inequalities proved above for Alice. In the remaining case $p=0$, each opening move at pool size $0$ is also available at pool size $2$, with two additional pool vertices in the successor state. The inequality for Alice therefore gives $\mathcal{B}(U,2) \ge \mathcal{B}(U,0)$. For the other direction, Alice can always choose a pool move at the state $(U,1)$, so $\mathcal{B}(U,2) = 1+\mathcal{A}(U,1) \le 1+\mathcal{B}(U,0)$.
\end{proof}

\begin{lemma}
  \label{lem:part-two}
  For every multiset $U$, $r \ge 1$, $p \ge 0$, and $\mathcal{T} \in \{\mathcal{A},\mathcal{B}\}$, it holds that $0 \le \mathcal{T}(U \cup \{r+2\},p)-\mathcal{T}(U \cup \{r\},p) \le 1$.
\end{lemma}

\begin{proof}
  We match each possible move with the corresponding move in the other game and use induction on the total number of uncolored vertices in the larger state, namely $p + (r + 2) + \sum_{s \in U} s$.
  In the base case $U = \emptyset$, $r = 1$, $p = 0$, it is easy to check that we have $\mathcal{T}(\{3\},p) = \mathcal{T}(\{1\},p) = 0$ for both players.
  
  Let $\mathcal{T}'$ be the function corresponding to the next player, i.e. $\mathcal{T} \neq \mathcal{T}'$.
  If we make a move in a part which differs between games, then we transition from $(U \cup \{r+2\},p)$ and $(U \cup \{r\},p)$ to $(U,p + r + 1)$ and $(U,p + r - 1)$, respectively. But from \Cref{lem:pool-two} we know that
  \begin{align*}
    0 \le \mathcal{T}'(U,p + r + 1) - \mathcal{T}'(U,p + r - 1) \le 1.
  \end{align*}
  
  Moreover, if we either start coloring another part or color a pool vertex, then we transition from $(U \cup \{r+2\},p)$ and $(U \cup \{r\},p)$ to $(U' \cup \{r+2\}, p')$ and $(U' \cup \{r\}, p')$, respectively. However, the total number of uncolored vertices in the larger state decreases by $1$, so we can apply the induction hypothesis and thus obtain
  \begin{align*}
    0 \le \mathcal{T}'(U' \cup \{r + 2\}, p') - \mathcal{T}'(U' \cup \{r\}, p') \le 1.
  \end{align*}
  Since $\mathcal{T}(U \cup \{r+2\},p)$ and $\mathcal{T}(U \cup \{r\},p)$ are just appropriate minima or maxima of the values of $\mathcal{T}'$ among all possible successor states identified above, their difference is also either $0$ or $1$.
\end{proof}

The following inequality is the key to restricting the choices among uncolored parts.

\begin{lemma}
  \label{lem:parity-exchange}
  For every multiset $U$, $r \ge 1$, $p \ge 0$, and $\mathcal{T} \in \{\mathcal{A},\mathcal{B}\}$, it holds that $\mathcal{T}(U \cup \{r+2\},p) \le \mathcal{T}(U \cup \{r\},p+2)$.
\end{lemma}

\begin{proof}
  Again we use induction on the common total number of uncolored vertices in these two states. Making the first move in the differing part in both games leads to identical successor states: we just remove these parts from the multiset and add $r + 1$ and $r - 1$ vertices to their respective pools. And exactly as in the proof of the lemma above, for every other opening move and for pool moves when $p>0$, the comparison of the values of the respective successor states follows from the induction hypothesis. Thus, for Bob, taking maxima preserves the inequality; an additional pool option on the right when $p=0$ can only increase that maximum. For Alice with $p>0$, all options are matched, so taking minima also preserves the inequality.

  It remains just to consider Alice's additional pool option on the right when $p=0$. By \Cref{cor:bob-pool} and \Cref{lem:part-two},
  \begin{align*}
    \mathcal{B}(U \cup \{r\},1) = 1+\mathcal{A}(U \cup \{r\},0) \ge \mathcal{A}(U \cup \{r+2\},0).
  \end{align*}
  The left-hand value cannot be larger than this extra option either, thus completing the proof.
\end{proof}

Finally, we can employ the lemmas presented above to ensure there are only a few moves that have to be considered by each player:
\begin{theorem}
  \label{thm:extreme-moves}
  Before the fixing move, the following restrictions preserve the values of the game:
  \begin{enumerate}
    \item Alice either starts a smallest uncolored part of one of the two parities, or plays in the pool using an existing color;
    \item Bob plays in the pool using a new color whenever $p>0$; if $p=0$, he starts a largest uncolored part of one of the two parities.
  \end{enumerate}
  In particular, at most three choices for Alice and at most two choices for Bob need to be considered.
\end{theorem}

\begin{proof}
  Let $a \le b$ be the sizes of two uncolored parts with the same parity, and let $U$ contain the sizes of all other uncolored parts. Repeated application of \Cref{lem:parity-exchange} for either player ensures that $\mathcal{T}(U \cup \{b\},p+a-1) \le \mathcal{T}(U \cup \{a\},p+b-1)$.
  Note that both states in this inequality are the successor states of $(U \cup \{a, b\}, p)$ after starting the smaller and the larger part, respectively.
  Since starting a part contributes no excess color, Alice, who minimizes the value, never needs to choose the larger part, and Bob, who maximizes it, never needs to choose the smaller one.

  To complete the proof we note that by \Cref{lem:numerical-state}, Alice only considers using an existing color when playing in the pool.
  Hence, Alice has at most three relevant choices: the smallest part of each parity and a pool move.

  On the other hand, by \Cref{cor:bob-pool}, whenever $p>0$, Bob has an optimal move in the pool using a new color.
  And in the case $p = 0$, it follows from the argument above that he has at most two choices, namely the largest part of each parity.
\end{proof}

\subsection{The dynamic program}

The theorem above restricts the number of choices, but $p$ could still be as large as the number of vertices. We now show that only $O(k)$ pool sizes need to be stored.
\begin{lemma}
  \label{lem:pool-saturation}
  Let $m=|U| \ge 1$. Then $\mathcal{A}(U,p) = m-1$ if $p \ge 2m-2$ and $\mathcal{B}(U,p) = m$ if $p \ge 2m-1$.
\end{lemma}

\begin{proof}
  If Alice proceeds by coloring a new part on each of her turns, then at most $m-1$ moves by Bob precede the fixing move, so $\mathcal{A}(U,p) \le m-1$.

  Assume $p \ge 2 m - 2$ and let Bob always choose a pool vertex. Before his $j$-th turn ($j = 1, \ldots, m - 1$), Alice and Bob have colored at most $j$ and $j - 1$ pool vertices, respectively. At least $p-j-(j-1) \ge 2m-2-2j+1 \ge 1$ pool vertices remain, and coloring vertices from other parts can only add vertices to the pool. Moreover, all these turns occur before the fixing move, since Bob has started no parts and Alice has made fewer than $m$ moves. Bob therefore makes at least $m-1$ pool moves, proving $\mathcal{A}(U,p) \ge m-1$.

  Finally, for $p \ge 2m-1$, by \Cref{cor:bob-pool} we get $\mathcal{B}(U,p)=1+\mathcal{A}(U,p-1)=m$.
\end{proof}

Now we can proceed with a proper representation of $U$ for our dynamic program. Let us sort the odd part sizes into a list $o_1 \le \cdots \le o_a$, and the even part sizes into a list $e_1 \le \cdots \le e_b$, where $a+b=k$. By \Cref{thm:extreme-moves}, whenever Alice starts a part she removes a left endpoint of one list, and whenever Bob starts a part he removes a right endpoint. Consequently, the sizes of uncolored parts are always represented by two intervals, $o_{i_o},\ldots,o_{j_o}$ and $e_{i_e},\ldots,e_{j_e}$, i.e. effectively by the fixed lists of sizes and moving endpoint indices $i_o,j_o,i_e,j_e$.
An interval with its left endpoint greater than its right endpoint is empty.

Let $D(i_o,j_o,i_e,j_e,p,t)$ be the value of the remaining excess-color game, where $t \in \{A,B\}$ specifies the player to move. For each state, let $U$ be the multiset represented by the two intervals and put $m=|U|$. We evaluate $D$ recursively, storing already computed values in an array indexed by these parameters, according to the following rules:
\begin{enumerate}
  \item If $m=0$, return $0$, regardless of $p$ or $t$.
  \item If $t=A$ and $p \ge 2m-2$, return $m-1$. If $t=B$ and $p \ge 2m-1$, return $m$. These checks are made before storing or expanding the state.
  \item Otherwise, consider the moves from \Cref{thm:extreme-moves}. Starting a part of size $r$ removes the appropriate interval endpoint and replaces $p$ by $p+r-1$, with contribution zero. A pool move replaces $p$ by $p-1$, with contribution zero for Alice and one for Bob. In every case the player to move changes.
  \item Return and store the minimum of the resulting contribution plus successor value when $t=A$, or the maximum when $t=B$.
\end{enumerate}
Then, by \Cref{thm:numerical-game}, the desired value of our original coloring game is equal to $k+D(1,a,1,b,0,A)$.

\begin{theorem}
  \label{thm:polynomial-algorithm}
  Given an explicit list of positive integers $r_1,\ldots,r_k$, the above algorithm computes $\chi_g(K_{r_1,\ldots,r_k})$ in $O(k^5)$ time and memory. Moreover, an optimal move for any player in any position can also be computed within these bounds from the numerical representation of this position.
\end{theorem}

\begin{proof}
  The recursive evaluations decrease the number of uncolored vertices. The terminal values follow from \Cref{eq:pool-terminal,lem:pool-saturation}, and \Cref{thm:extreme-moves} shows that the retained options attain the minima and maxima in \Cref{eq:pool-alice,eq:pool-bob}. Induction on the number of uncolored vertices therefore proves that every returned value is the corresponding $\mathcal{A}(U,p)$ or $\mathcal{B}(U,p)$. The answer for the initial position follows from \Cref{thm:numerical-game}.

  There are $O((a+1)^2(b+1)^2)=O(k^4)$ pairs of intervals. Every stored state has $p<2m-1 \le 2k-1$, and there are two possible players to move. Thus there are $O(k^5)$ stored states. Each state has at most three successor states, including those handled immediately by the terminal rules, so the number of arithmetic operations and memory reads/writes is $O(k^5)$. The initial sorting of part sizes takes $O(k\log k)$ comparisons and does not change this bound.

  To recover an optimal move, store an option attaining the minimum or maximum along with each value. In a state covered by \Cref{lem:pool-saturation}, Alice may start any remaining part, and Bob may choose a pool vertex. These moves attain the stated values by the proof of that lemma. If an opponent starts an interior part of one of the parity lists, players may form new sorted lists of the remaining uncolored parts and apply the same algorithm to the observed position. The accumulated colors and the unavoidable opening colors are fixed at that position, so minimizing or maximizing the remaining excess value gives an optimal move in the original game. Thus the strategy does not require the opponent to follow the restricted choices. Once all parts have been started, existing colors suffice to finish.

  Finally, we observe that the stored values lie between $0$ and $k$: Alice can start a part on every turn, allowing at most $|U|-1$ Bob moves when she moves first and at most $|U|$ when Bob moves first. Stored pool sizes are $O(k)$. Only the input sizes and the intermediate updates $p+r-1$ may require more bits, and their bit lengths are bounded by the input length plus $O(\log k)$. Large pool values are handled immediately by \Cref{lem:pool-saturation}. Consequently, there is no need to expand parts into individual vertices, and the total complexity remains polynomial.
\end{proof}

The input convention in \Cref{thm:polynomial-algorithm} assumes all $k$ parts are listed explicitly, even when their sizes $r_i$ are encoded in binary (e.g. by providing the part sizes), rather than in unary (e.g. given as an explicit graph). However, the above theorem does not establish polynomial running time for an encoding that describes exponentially many repeated parts by their multiplicities written in binary.

The input convention in \Cref{thm:polynomial-algorithm} assumes all $k$ parts are listed explicitly, even when their sizes are encoded in binary. The bound does not assert polynomial time for an encoding that describes exponentially many repeated parts by their multiplicities written in binary.

\section*{Acknowledgements}
During the preparation of this work, the authors used OpenAI ChatGPT for verification
of some of the proofs. After using this tool, the authors amended proofs
as needed and take full responsibility for their contents.
Generative AI was also used to find some relevant prior art references and to improve the writing (including spelling
and grammar) for this paper.
In addition, the authors also used generative AI to assist in creating software for verification of small cases.

\bibliography{games-multipartite.bib}

\end{document}